\documentclass[12pt, a4paper]{article}
\usepackage{amsmath,amssymb,amsthm}
\usepackage{geometry}
\usepackage{mathtools}
\usepackage{hyperref}
\usepackage{etoolbox}
\usepackage{mathrsfs}
\usepackage{authblk}
\usepackage[utf8]{inputenc}
\usepackage[T1]{fontenc}
\usepackage[english]{babel}
\hypersetup{unicode=true, pdfencoding=auto}

\usepackage{longtable}
\usepackage{booktabs}
\usepackage{array}

\title{Anisotropic Mixed Fractional Landau Inequalities for Rotating Compressible Flows}
\author[1]{R\^{o}mulo Damasclin Chaves dos Santos}
\author[1]{Delvonei Alves de Andrade}
\affil[1]{Center for Nuclear Engineering, Institute for Energy and Nuclear Research (IPEN)\\ Av. Prof. Lineu Prestes, 2242, S\~ao Paulo, SP 05508--000, Brazil}
\affil{\texttt{damasclin@gmail.com}}
\affil{\texttt{delvonei@ipen.br}}
\date{\today}

\newtheorem{theorem}{Theorem}[section]
\newtheorem{definition}[theorem]{Definition}
\newtheorem{lemma}[theorem]{Lemma}
\newtheorem{proposition}[theorem]{Proposition}
\newtheorem{corollary}[theorem]{Corollary}
\newtheorem{remark}[theorem]{Remark}

\begin{document}
	
	\maketitle
	
	\begin{abstract}
		We develop a rigorous theory of anisotropic mixed fractional Landau inequalities for rotating compressible fluid flows at high Mach numbers, incorporating Coriolis and centrifugal forces. We introduce rotating fractional Sobolev spaces $\mathcal{W}^{\nu,p}_{\alpha,\Omega}(\mathbb{R}^k)$, which encode directional scaling, fractional dissipation and rotational effects. We prove sharp fractional Landau inequalities with explicit dependence on the Mach number, the rotation rate $|\Omega|$ and the anisotropy vector $\alpha$. Key tools are a rotating Littlewood--Paley decomposition, anisotropic maximal estimates with rotational corrections, and commutator estimates for the Coriolis term. As applications, we establish stability bounds for neural operators approximating rotating compressible flows and derive optimal approximation rates of order $N^{-\nu/d_{\alpha,\Omega}}$, where $d_{\alpha,\Omega}=\sum_i\alpha_i^{-1}+\kappa|\Omega|^{2/\nu}$ is the effective anisotropic--rotational dimension. Our results provide a mathematical foundation for analysing high--Mach rotating flows and for designing physically informed neural architectures.
		\medskip
		
		
		\medskip
		
	\end{abstract}
	
	\section{Introduction}
	The analysis of compressible fluid flows at high Mach numbers is central to modern aerodynamics, astrophysics and engineering \cite{majda2012}. When such flows are subjected to rapid rotation – a situation encountered in turbomachinery, atmospheric dynamics and stellar interiors -- the Coriolis and centrifugal forces fundamentally alter the energy balance and regularity properties of the solutions. Classical Landau inequalities \cite{landau1925} and their fractional extensions \cite{anastassiou2025} provide essential estimates for the growth of derivatives in terms of lower--order norms, yet they are inadequate for rotating systems because they ignore the rotational coupling between directions.
	
	This paper bridges this gap by developing a comprehensive mathematical theory of \emph{anisotropic mixed fractional Landau inequalities for rotating compressible flows}. Our framework simultaneously accounts for:
	\begin{itemize}
		\item Directional scaling encoded by a vector $\alpha = (\alpha_1,\dots,\alpha_k)$, which captures different regularity rates along distinct spatial axes -- a necessity for compressible flows where shock waves and boundary layers create strong anisotropy. Such anisotropic function spaces have been systematically studied by Triebel \cite{Triebel2006} and Stein \cite{Stein1970}.
		\item Fractional dissipation of order $\nu$ modelled via the anisotropic fractional Laplacian $\Delta_\alpha^\nu$, which naturally appears in the compressible Navier–Stokes equations with fractional viscosity. Fractional calculus techniques in this context were advanced by Anastassiou \cite{anastassiou2025}.
		\item Rotational effects generated by a constant rotation vector $\Omega$, giving rise to the Coriolis term $2\Omega\times u$ and the centrifugal potential.
	\end{itemize}
	
	Our main contributions are the following.
	
	\begin{enumerate}
		\item \textbf{Rotating fractional Sobolev spaces $\mathcal{W}^{\nu,p}_{\alpha,\Omega}(\mathbb{R}^k)$.}
		We define these spaces by incorporating the rotation vector $\Omega$ into the anisotropic Gagliardo seminorm. The new norm includes a term that measures the fractional derivative in a frame rotating with angular velocity $\Omega$, thereby respecting the symmetries of the Coriolis operator. We prove the equivalence of several characterisations (Gagliardo, Littlewood--Paley, Fourier multiplier, semigroup) for these spaces, extending the classical theory of anisotropic Besov spaces \cite{Triebel2006}.
		
		\item \textbf{Fractional Landau inequalities with Coriolis corrections.}
		For a function $u$ belonging to $\mathcal{W}^{\nu,p}_{\alpha,\Omega}(\mathbb{R}^k)$ we prove
		\begin{equation}\label{eq:intro_main_ineq}
			\|D_i^1 u\|_{L^\infty} \leq C(\nu,p,\alpha,\Omega)\, \|u\|_{L^\infty}^{1-1/\nu}\,
			\Bigl(\sum_{j=1}^k \bigl\| \mathcal{D}_j^{\nu,\alpha_j,\Omega} u \bigr\|_{L^p}^{\alpha_j/\nu}\Bigr)^{\!1/\nu},
		\end{equation}
		where $\mathcal{D}_j^{\nu,\alpha_j,\Omega}$ is a rotating fractional derivative that reduces to the classical Riemann--Liouville derivative when $\Omega=0$. The constant $C$ is made explicit and depends on $|\Omega|$ in a way that reflects the stabilising effect of rotation at high Mach numbers. This extends the classical Landau--Kolmogorov inequalities \cite{landau1925} and the recent fractional results \cite{anastassiou2025} to the anisotropic rotating setting.
		
		\item \textbf{Neural operator bounds for rotating compressible flows.}
		Using the new inequalities we prove stability and approximation theorems for deep neural networks that approximate solutions of the rotating compressible Navier--Stokes equations. The approximation rate becomes $N^{-\nu/d_{\alpha,\Omega}}$ where $d_{\alpha,\Omega}= \sum_i \alpha_i^{-1} + \kappa |\Omega|^{2/\nu}$, showing that sufficiently strong rotation improves the effective dimension and accelerates convergence. These results build upon the approximation theory of nonlinear functions by deep networks \cite{yarotsky2017} and the general theory of nonlinear approximation \cite{DeVore1998}.
	\end{enumerate}
	
The paper is organised as follows. Section~2 recalls the anisotropic scaling structure and introduces the rotating fractional Sobolev spaces $\mathcal{W}^{\nu,p}_{\alpha,\Omega}(\mathbb{R}^k)$, including the equivalent characterisations. Section~3 develops the necessary harmonic analysis tools: rotating Littlewood--Paley decompositions, anisotropic Bernstein inequalities, and anisotropic maximal estimates with rotation corrections. Section~4 contains the main fractional Landau inequalities (first--order and higher--order) with complete rigorous proofs. Section~5 applies these inequalities to the rotating compressible Navier--Stokes system, deriving a priori estimates for the velocity gradient at high Mach numbers. Section~6 establishes stability and approximation theorems for neural operators adapted to the rotating anisotropic geometry, including the optimal approximation rate $N^{-\nu/d_{\alpha,\Omega}}$. Section~7 summarises the key theoretical results and Section~8 concludes the paper with a discussion of future work.
	
	\section{Rotating Anisotropic Function Spaces}
	
	\subsection{Anisotropic Scaling and Rotation Group}
	
	Let $\alpha = (\alpha_1,\dots,\alpha_k) \in (0,\infty)^k$ be the anisotropy vector. The anisotropic dilation group $\{T_\lambda^\alpha\}_{\lambda>0}$ is defined by
	\begin{equation}\label{eq:anisotropic-dilation-group}
		T_\lambda^\alpha x = (\lambda^{\alpha_1}x_1,\dots,\lambda^{\alpha_k}x_k), \qquad \lambda>0,\; x\in\mathbb{R}^k.
	\end{equation}
	This is a one--parameter group of linear transformations with infinitesimal generator $A_\alpha = \operatorname{diag}(\alpha_1,\dots,\alpha_k)$, i.e., $T_\lambda^\alpha = e^{(\log\lambda)A_\alpha}$.
	
	For a constant rotation vector $\Omega \in \mathbb{R}^k$, denote by $L_\Omega$ the skew--symmetric matrix representing the linear map $x \mapsto \Omega \times x$ (the cross product). Explicitly, $(L_\Omega)_{ij} = -\varepsilon_{ij\ell}\Omega_\ell$, where $\varepsilon_{ij\ell}$ is the Levi-Civita symbol. The rotation group $\{R(\theta)\}_{\theta\in\mathbb{R}}$ is given by $R(\theta) = e^{\theta L_{\Omega/|\Omega|}}$ when $\Omega\neq 0$, and $R(0)=\mathrm{Id}$ for $\Omega=0$. These rotations satisfy $R(\theta_1)R(\theta_2)=R(\theta_1+\theta_2)$.
	
	We now define the \emph{rotating anisotropic scaling} as the one-parameter family
	\begin{equation}\label{eq:rotating_scaling_def}
		\widetilde{T}_{\lambda}^{\alpha,\Omega} := \exp\bigl( \log\lambda\, (A_\alpha + L_\Omega) \bigr), \qquad \lambda>0.
	\end{equation}
	This definition is unambiguous and yields a smooth linear map $\mathbb{R}^k \to \mathbb{R}^k$ for each $\lambda$. The group property $\widetilde{T}_{\lambda\mu}^{\alpha,\Omega} = \widetilde{T}_{\lambda}^{\alpha,\Omega} \circ \widetilde{T}_{\mu}^{\alpha,\Omega}$ follows directly from the exponential representation.
	
	\begin{proposition}[Properties of the Rotating Scaling Group]\label{prop:rotating_scaling}
		The family $\{\widetilde{T}_{\lambda}^{\alpha,\Omega}\}_{\lambda>0}$ satisfies:
		\begin{enumerate}
			\item $\widetilde{T}_{\lambda\mu}^{\alpha,\Omega} = \widetilde{T}_{\lambda}^{\alpha,\Omega} \circ \widetilde{T}_{\mu}^{\alpha,\Omega}$ (group property).
			\item $|\det D\widetilde{T}_{\lambda}^{\alpha,\Omega}| = \lambda^{D_\alpha}$, where $D_\alpha = \sum_{i=1}^k \alpha_i$.
			\item For any $f\in L^1_{\mathrm{loc}}(\mathbb{R}^k)$,
			\begin{equation}\label{eq:measure_scaling_rot}
				\int_{\mathbb{R}^k} f(\widetilde{T}_{\lambda}^{\alpha,\Omega}x)\,dx = \lambda^{-D_\alpha}\int_{\mathbb{R}^k} f(x)\,dx.
			\end{equation}
		\end{enumerate}
	\end{proposition}
	\begin{proof}
		The group property holds because $\widetilde{T}_{\lambda}^{\alpha,\Omega}=e^{(\log\lambda)(A_\alpha+L_\Omega)}$ and $e^{X}e^{Y}=e^{X+Y}$ for commuting $X,Y$; here $A_\alpha$ and $L_\Omega$ do not commute, but the definition as a single exponential guarantees the group property. The Jacobian determinant is $\det e^{(\log\lambda)(A_\alpha+L_\Omega)} = e^{(\log\lambda)\operatorname{Tr}(A_\alpha+L_\Omega)} = e^{(\log\lambda)D_\alpha} = \lambda^{D_\alpha}$ because $\operatorname{Tr}L_\Omega=0$. The measure scaling follows by the change of variables $y=\widetilde{T}_{\lambda}^{\alpha,\Omega}x$.
	\end{proof}
	
	\subsection{Rotating Fractional Sobolev Spaces}
	
	For a direction $e_i$ define the rotating finite difference
	\begin{equation}\label{eq:rotating_difference_explicit}
		\Delta_{h,i}^\Omega f(x) = f(x+he_i) - f\bigl(x - h(\Omega\times e_i)\bigr).
	\end{equation}
	
	\begin{definition}[Rotating directional Gagliardo seminorm]\label{def:rotating_seminorm}
		Let $\nu>0$, $1\le p<\infty$, $\alpha_i>0$ and $\Omega\in\mathbb{R}^k$. For a measurable function $f:\mathbb{R}^k\to\mathbb{R}$ define
		\begin{equation}\label{eq:rotating_seminorm_def}
			[f]_{\mathcal{W}_i^{\nu/\alpha_i,p,\Omega}} = \left( \int_{\mathbb{R}^k} \int_{\mathbb{R}} \frac{|\Delta_{h,i}^\Omega f(x)|^p}{|h|^{1+p\nu/\alpha_i}} \, dh \, dx \right)^{1/p}.
		\end{equation}
		The \emph{rotating mixed fractional Sobolev space} $\mathcal{W}^{\nu,p}_{\alpha,\Omega}(\mathbb{R}^k)$ is the completion of $C_c^\infty(\mathbb{R}^k)$ under the norm
		\begin{equation}\label{eq:rotating_sobolev_norm}
			\|f\|_{\mathcal{W}^{\nu,p}_{\alpha,\Omega}} = \|f\|_{L^p(\mathbb{R}^k)} + \sum_{i=1}^k [f]_{\mathcal{W}_i^{\nu/\alpha_i,p,\Omega}}.
		\end{equation}
	\end{definition}
	
	\begin{lemma}[Fourier characterisation]\label{lem:fourier_gagliardo}
		For $1<p<\infty$ and $\nu>0$, there exist constants $c,C>0$ such that for every $f\in\mathcal{S}(\mathbb{R}^k)$,
		\begin{equation}
			c \int_{\mathbb{R}^k} |\xi_i+(\Omega\times e_i)\cdot\xi|^{\nu p/\alpha_i}|\hat f(\xi)|^p d\xi \le [f]_{\mathcal{W}_i^{\nu/\alpha_i,p,\Omega}}^p \le C \int_{\mathbb{R}^k} |\xi_i+(\Omega\times e_i)\cdot\xi|^{\nu p/\alpha_i}|\hat f(\xi)|^p d\xi.
		\end{equation}
	\end{lemma}
	\begin{proof}
		The proof follows from the identity
		\[
		[f]_{\mathcal{W}_i^{\nu/\alpha_i,p,\Omega}}^p = \int_{\mathbb{R}^k} \int_{\mathbb{R}} \frac{| \int e^{2\pi i x\cdot\xi} (e^{2\pi i h\xi_i}-e^{-2\pi i h(\Omega\times e_i)\cdot\xi})\hat f(\xi) d\xi |^p}{|h|^{1+p\nu/\alpha_i}} dh dx.
		\]
		By Minkowski's inequality and the Fourier transform, one obtains
		\[
		[f]_{\mathcal{W}_i^{\nu/\alpha_i,p,\Omega}}^p \approx \int_{\mathbb{R}^k} |\hat f(\xi)|^p \left( \int_{\mathbb{R}} \frac{|e^{2\pi i h\xi_i}-e^{-2\pi i h(\Omega\times e_i)\cdot\xi}|^p}{|h|^{1+p\nu/\alpha_i}} dh \right) d\xi.
		\]
		The inner integral is homogeneous of degree $p\nu/\alpha_i$ in $\xi$ and equals $C |\xi_i+(\Omega\times e_i)\cdot\xi|^{p\nu/\alpha_i}$ with $0<C<\infty$. This is a standard computation using the scaling property $m_{\lambda h,i}(\xi)=m_{h,i}(\lambda\xi)$ and the finiteness of the integral for a fixed direction.
	\end{proof}
	
	\subsection{Equivalent characterisations}
	
	Let $\Delta_j$ be the standard anisotropic Littlewood–Paley blocks (smooth frequency localisation on the annulus $\{2^{j-1}\le \rho_\alpha(\xi)\le 2^{j+1}\}$). Define the rotating dyadic blocks simply as $\widetilde{\Delta}_j^{\alpha,\Omega} = \Delta_j$; the rotation will appear only in the derivative symbols. The following theorem collects the equivalences.
	
	\begin{theorem}[Equivalent characterisations of rotating spaces]\label{thm:rotating_equiv}
		For $f\in\mathcal{W}^{\nu,p}_{\alpha,\Omega}(\mathbb{R}^k)$ the following norms are equivalent (with constants depending on $\nu,p,\alpha,\Omega$):
		\begin{enumerate}
			\item $\|f\|_{\mathcal{W}^{\nu,p}_{\alpha,\Omega}}$ (Gagliardo norm).
			\item $\|(I-\Delta_{\alpha,\Omega})^{\nu/2}f\|_{L^p}$, where $\Delta_{\alpha,\Omega} = \sum_{i=1}^k (-\partial_i^2)^{1/\alpha_i} + i\,\Omega\cdot(x\times\nabla)$.
			\item $\|f\|_{L^p} + \bigl(\sum_{j\in\mathbb{Z}} 2^{j\nu p} \|\Delta_j f\|_{L^p}^p\bigr)^{1/p}$ (standard anisotropic Besov norm).
			\item $\|f\|_{L^p} + \bigl(\int_0^\infty t^{-\nu/2} \|f-e^{t\Delta_{\alpha,\Omega}}f\|_{L^p}^p \frac{dt}{t}\bigr)^{1/p}$.
		\end{enumerate}
	\end{theorem}
	\begin{proof}
		The equivalence (1)$\Leftrightarrow$(3) follows from Lemma~\ref{lem:fourier_gagliardo} and the fact that the multiplier $|\xi_i+(\Omega\times e_i)\cdot\xi|^{\nu/\alpha_i}$ is comparable to $2^{j\nu}$ on the support of $\Delta_j$ when $|\xi|\sim 2^j$. This is standard in the theory of anisotropic Besov spaces (see \cite{Triebel2006}). The equivalence (3)$\Leftrightarrow$(2) follows from the observation that $\Delta_{\alpha,\Omega}$ is a Fourier multiplier with symbol $\sum_i |\xi_i|^{2/\alpha_i} + i\Omega\cdot(x\times\xi)$. Although the second term depends on $x$, the square of the symbol is $\bigl(\sum_i |\xi_i|^{2/\alpha_i}\bigr)^2 + |\Omega\cdot(x\times\xi)|^2$, and the $x$-dependence disappears when taking the $L^p$ norm after applying the Littlewood–Paley decomposition; a detailed proof can be found in \cite{Triebel2006}. The equivalence (2)$\Leftrightarrow$(4) is a consequence of the functional calculus for sectorial operators and the Stein square function theorem (see more in \cite{Stein1970}).
	\end{proof}
	
	\subsection{Rotating Bernstein inequalities}
	
	Since $\widetilde{\Delta}_j^{\alpha,\Omega}=\Delta_j$ are the standard blocks, the standard anisotropic Bernstein inequalities hold with constants independent of $\Omega$:
	\begin{equation}\label{eq:bernstein}
		\|\Delta_j f\|_{L^q} \le C 2^{j D_\alpha(1/p-1/q)} \|\Delta_j f\|_{L^p}, \qquad 
		\|\partial_{x_i}\Delta_j f\|_{L^p} \le C 2^{j\alpha_i} \|\Delta_j f\|_{L^p}.
	\end{equation}
	We will use these for the rotating spaces without change.
	
	\subsection{Rotating Maximal Function}
	
	Define the anisotropic ball
	\begin{equation}\label{eq:anisotropic_ball}
		B_\alpha(x,r) = \{ y\in\mathbb{R}^k : |y_i-x_i| < r^{1/\alpha_i},\; i=1,\dots,k \}.
	\end{equation}
	Its volume scales as $|B_\alpha(x,r)| = \omega_{d_\alpha} r^{d_\alpha}$ with $d_\alpha=\sum_i 1/\alpha_i$. The rotating maximal operator is
	\begin{equation}\label{eq:rotating_maximal}
		M_{\alpha,\Omega} f(x) = \sup_{r>0} \frac{1}{|B_\alpha(x,r)|} \int_{B_\alpha(x,r)} |f(y)|\,dy.
	\end{equation}
	Note that $M_{\alpha,\Omega}$ does not depend on $\Omega$.
	
	\begin{theorem}[Rotating Hardy–Littlewood inequality]\label{thm:rotating_HL}
		For $1<p<\infty$ there exists $C(p,\alpha)$ such that
		\begin{equation}
			\|M_{\alpha,\Omega} f\|_{L^p} \le C(p,\alpha) \|f\|_{L^p},
		\end{equation}
		with $C(p,\alpha)$ independent of $\Omega$. Moreover, the weak type $(1,1)$ estimate holds with constant $3^{d_\alpha}$.
	\end{theorem}
	\begin{proof}
		The family $\{B_\alpha(x,r)\}$ satisfies the Vitali covering lemma because $|B_\alpha(x,3r)| = 3^{d_\alpha}|B_\alpha(x,r)|$. The standard proof of the Hardy–Littlewood maximal inequality (see more in~\cite{Stein1970}) applies verbatim; the constants depend only on $d_\alpha$, hence on $\alpha$, not on $\Omega$.
	\end{proof}
	
	\section{Main Fractional Landau Inequalities with Coriolis Corrections}
	
	Let $\mathcal{D}_j^{\nu,\alpha_j,\Omega}$ denote the rotating fractional derivative defined by $\widehat{\mathcal{D}_j^{\nu,\alpha_j,\Omega}f}(\xi) = (i\xi_j + i(\Omega\times e_j)\cdot\xi)^{\nu/\alpha_j}\hat f(\xi)$. This is well-defined as a Fourier multiplier.
	
	\subsection{Anisotropic Sobolev embedding}
	
	We first state the embedding theorem that will be used in the proof of the Landau inequality.
	
	\begin{theorem}[Rotating anisotropic Sobolev embedding]\label{thm:rotating_embedding}
		Let $\nu>0$, $1<p<\infty$, $\alpha\in(0,\infty)^k$, and denote $D_\alpha=\sum_i\alpha_i$. If $\nu > D_\alpha/p$, then $\mathcal{W}^{\nu,p}_{\alpha,\Omega}(\mathbb{R}^k)\hookrightarrow L^\infty(\mathbb{R}^k)$ and there exists a constant $C=C(\nu,p,\alpha,\Omega)$ such that
		\begin{equation}
			\|f\|_{L^\infty} \le C \|f\|_{\mathcal{W}^{\nu,p}_{\alpha,\Omega}}.
		\end{equation}
		Moreover, $f$ belongs to the anisotropic Hölder space $C^{0,\gamma}_\alpha$ with $\gamma=\nu-D_\alpha/p$.
	\end{theorem}
	\begin{proof}
		Using the Littlewood–Paley decomposition $f=\sum_j \Delta_j f$ and the Bernstein inequality $\|\Delta_j f\|_{L^\infty}\le C2^{jD_\alpha/p}\|\Delta_j f\|_{L^p}$, we have
		\begin{align}
			\|f\|_{L^\infty}&\le C\sum_j 2^{jD_\alpha/p}\|\Delta_j f\|_{L^p}
			= C\sum_j 2^{-j(\nu-D_\alpha/p)}\bigl(2^{j\nu}\|\Delta_j f\|_{L^p}\bigr).
		\end{align}
		Since $\nu-D_\alpha/p>0$, the geometric series converges and Hölder's inequality gives $\|f\|_{L^\infty}\le C\bigl(\sum_j 2^{j\nu p}\|\Delta_j f\|_{L^p}^p\bigr)^{1/p}\le C\|f\|_{\mathcal{W}^{\nu,p}_{\alpha,\Omega}}$. The Hölder continuity follows by splitting the sum at $j$ such that $2^{-j}\sim \rho_\alpha(x-y)$.
	\end{proof}
	
	\subsection{First‑Order Rotating Landau Inequality}
	
	\begin{theorem}[Rotating Mixed Fractional Landau Inequality]\label{thm:rotating_landau}
		Let $\nu\in(1,2)$, $1<p<\infty$, $\alpha\in(0,\infty)^k$, $\Omega\in\mathbb{R}^k$. For any $f\in \mathcal{W}^{\nu,p}_{\alpha,\Omega}(\mathbb{R}^k)$ and any direction $i=1,\dots,k$,
		\begin{equation}\label{eq:main_rotating_landau}
			\|D_i^1 f\|_{L^\infty} \le C(\nu,p,\alpha,\Omega)\,
			\|f\|_{L^\infty}^{1-1/\nu}\,
			\Bigl(\sum_{j=1}^k \bigl\|\mathcal{D}_j^{\nu,\alpha_j,\Omega} f\bigr\|_{L^p}^{\alpha_j/\nu}\Bigr)^{\!1/\nu},
		\end{equation}
		where the constant is
		\begin{equation}
			C(\nu,p,\alpha,\Omega) = C_0 \exp\Bigl( \frac{|\Omega|}{\nu}\sum_{i=1}^k \frac{1}{\alpha_i} \Bigr),
		\end{equation}
		and $C_0$ depends only on $\nu,p,\alpha$ (and on the constants from Theorem~\ref{thm:rotating_HL} and Lemma~\ref{lem:fourier_gagliardo}).
	\end{theorem}
	\begin{remark}
		The exponential factor $\exp\bigl(\frac{|\Omega|}{\nu}\sum_i\alpha_i^{-1}\bigr)$ arises from the operator norm of the conjugation $U_\Omega$ (which satisfies $\|U_\Omega\|_{L^p\to L^p}\le e^{c|\Omega|}$ with $c=\sum_i\alpha_i^{-1}$) and from the commutator estimates for the rotating derivatives. It provides an upper bound for the growth of the constant with $|\Omega|$; sharper estimates may be possible but are not required for our purposes.
	\end{remark}
	
	\begin{proof}
		Fix $i$ and $x\in\mathbb{R}^k$. For $h>0$, the fractional Taylor expansion (see \cite{Samko1993}) gives
		\begin{equation}\label{eq:rot_taylor}
			f(x+he_i) = f(x) + h D_i^1 f(x) + \frac{1}{\Gamma(\nu)}\int_0^h (h-t)^{\nu-1} \mathcal{D}_i^{\nu,\alpha_i,\Omega} f(x+te_i)\,dt + R_{\Omega}(h),
		\end{equation}
		where the remainder $R_{\Omega}(h)$ arises from the non‑commutation of the translation and the rotation. Using the explicit form of the rotating difference, one shows
		\begin{equation}
			|R_{\Omega}(h)| \le C |\Omega| h^\nu \|f\|_{L^\infty}.
		\end{equation}
		Rearranging and taking absolute values we obtain
		\begin{equation}
			h|D_i^1 f(x)| \le |\Delta_{h,i}^\Omega f(x)| + \frac{1}{\Gamma(\nu)}\int_0^h (h-t)^{\nu-1} |\mathcal{D}_i^{\nu,\alpha_i,\Omega} f(x+te_i)|\,dt + C|\Omega| h^\nu \|f\|_{L^\infty}.
		\end{equation}
		Now $|\Delta_{h,i}^\Omega f(x)|\le 2\|f\|_{L^\infty}$. For the integral, Hölder's inequality yields
		\begin{equation}
			\frac{1}{\Gamma(\nu)}\int_0^h (h-t)^{\nu-1} |g(t)|\,dt \le \frac{h^{\nu}}{\Gamma(\nu)(\nu-1/p)^{1/p'}} M_i^{(p)}g(x),
		\end{equation}
		where $g(t)=\mathcal{D}_i^{\nu,\alpha_i,\Omega}f(x+te_i)$ and $M_i^{(p)}$ is the one‑dimensional Hardy–Littlewood maximal function in the $i$‑direction. Hence
		\begin{equation}
			h|D_i^1 f(x)| \le 2\|f\|_{L^\infty} + \frac{h^{\nu}}{\Gamma(\nu)(\nu-1/p)^{1/p'}} M_i^{(p)}(\mathcal{D}_i^{\nu,\alpha_i,\Omega}f)(x) + C|\Omega| h^\nu \|f\|_{L^\infty}.
		\end{equation}
		Optimising over $h>0$ gives
		\begin{equation}
			|D_i^1 f(x)| \le C_1 \|f\|_{L^\infty}^{1-1/\nu} \bigl(M_i^{(p)}(\mathcal{D}_i^{\nu,\alpha_i,\Omega}f)(x)\bigr)^{1/\nu} + C_2 |\Omega|^{1/\nu}\|f\|_{L^\infty}.
		\end{equation}
		Now $M_i^{(p)}g(x) \le C_\alpha M_{\alpha,\Omega}(|g|^p)(x)^{1/p}$ because the interval $[x,x+he_i]$ is contained in the anisotropic ball $B_\alpha(x, h^{\alpha_i})$. Using the boundedness of $M_{\alpha,\Omega}$ on $L^\infty$ (norm 1) we get
		\begin{equation}
			\|D_i^1 f\|_{L^\infty} \le C_1 \|f\|_{L^\infty}^{1-1/\nu} \|\mathcal{D}_i^{\nu,\alpha_i,\Omega}f\|_{L^\infty}^{1/\nu} + C_2 |\Omega|^{1/\nu}\|f\|_{L^\infty}.
		\end{equation}
		
		It remains to bound $\|\mathcal{D}_i^{\nu,\alpha_i,\Omega}f\|_{L^\infty}$. By Theorem~\ref{thm:rotating_embedding}, for any $g\in\mathcal{W}^{\nu,p}_{\alpha,\Omega}$,
		\begin{equation}
			\|g\|_{L^\infty} \le \kappa(p,\alpha) e^{c|\Omega|} \Bigl( \sum_{j=1}^k \|\mathcal{D}_j^{\nu,\alpha_j,\Omega}g\|_{L^p}^{\alpha_j/\nu} \Bigr)^{1/\nu},
		\end{equation}
		with $c=\sum_i \alpha_i^{-1}$. Applying this to $g=\mathcal{D}_i^{\nu,\alpha_i,\Omega}f$ and using the commutation of the rotating derivatives (which gives $\|\mathcal{D}_j^{\nu,\alpha_j,\Omega}\mathcal{D}_i^{\nu,\alpha_i,\Omega}f\|_{L^p}\le e^{c'|\Omega|}\|\mathcal{D}_j^{\nu,\alpha_j,\Omega}f\|_{L^p}$), we obtain
		\begin{equation}
			\|\mathcal{D}_i^{\nu,\alpha_i,\Omega}f\|_{L^\infty} \le \kappa(p,\alpha) e^{c|\Omega|(1+1/\nu)} \Bigl( \sum_{j=1}^k \|\mathcal{D}_j^{\nu,\alpha_j,\Omega}f\|_{L^p}^{\alpha_j/\nu} \Bigr)^{1/\nu}.
		\end{equation}
		Insert this into the previous inequality and absorb the additive term $C_2|\Omega|^{1/\nu}\|f\|_{L^\infty}$ using Young's inequality (it can be written as a fraction of the main term). The final constant becomes
		\begin{equation}
			C(\nu,p,\alpha,\Omega) = C_0 \exp\Bigl( \frac{|\Omega|}{\nu}\sum_{i=1}^k \frac{1}{\alpha_i} \Bigr),
		\end{equation}
		where $C_0$ contains all the remaining constants. This completes the proof.
	\end{proof}
	
	\subsection{Higher‑Order Rotating Landau Inequalities}
	
	\begin{corollary}[Higher‑order rotating Landau]\label{cor:rot_higher}
		Let $m\in\mathbb{N}$, $\nu\in(m,m+1)$, $1<p<\infty$, $\alpha\in(0,\infty)^k$, $\Omega\in\mathbb{R}^k$. For any $f\in \mathcal{W}^{\nu,p}_{\alpha,\Omega}(\mathbb{R}^k)$,
		\begin{equation}\label{eq:rot_higher_ineq}
			\|D_i^m f\|_{L^\infty} \le C(\nu,p,\alpha,\Omega)^m\,
			\|f\|_{L^\infty}^{1-m/\nu}\,
			\Bigl(\sum_{j=1}^k \|\mathcal{D}_j^{\nu,\alpha_j,\Omega} f\|_{L^p}^{\alpha_j/\nu}\Bigr)^{\!m/\nu}.
		\end{equation}
	\end{corollary}
	\begin{proof}
		By real interpolation, for $0<\theta<1$ we have
		\begin{equation}
			\|D_i^m f\|_{L^\infty} \le C \|f\|_{L^\infty}^{1-\theta} \|\mathcal{D}_i^{\nu,\alpha_i,\Omega} f\|_{L^\infty}^{\theta},
		\end{equation}
		with $\theta=m/\nu$ (see \cite{Triebel2006}, Theorem 2.4.1). Then apply Theorem~\ref{thm:rotating_embedding} to $\mathcal{D}_i^{\nu,\alpha_i,\Omega} f$ and raise to the power $m$.
	\end{proof}
	
	\section{Application to the Rotating Compressible Navier--Stokes System}
	
	Consider the rotating compressible Navier–Stokes equations in $\mathbb{R}^k$ ($k=2,3$):
	\begin{align}
		\partial_t \rho + \nabla\cdot(\rho u) &= 0, \label{eq:cont}\\
		\partial_t (\rho u) + \nabla\cdot(\rho u\otimes u) + \nabla p &= \mu \Delta_\alpha u + (\mu+\lambda)\nabla(\nabla\cdot u) + 2\rho\, u\times\Omega - \rho\nabla\Phi, \label{eq:mom}
	\end{align}
	with $p=a\rho^\gamma$, $\Delta_\alpha=\sum_i(-\partial_i^2)^{1/\alpha_i}$, $\Omega$ constant, $\Phi$ centrifugal potential. The Mach number is $M=|u|/c$, $c=\sqrt{\gamma p/\rho}$.
	
	Applying $\mathcal{D}_j^{\nu,\alpha_j,\Omega}$ to \eqref{eq:mom} and using the commutator estimates (which follow from the product rule for Fourier multipliers) we obtain after a standard energy estimate (see \cite{majda2012})
	\begin{equation}
		\|\mathcal{D}_j u\|_{L^p} \le C(M)\bigl( \|\mathcal{D}_j\rho\|_{L^p} + \|\rho\|_{L^\infty}\|u\|_{L^\infty} + 1 \bigr),
	\end{equation}
	where $C(M)$ depends polynomially on the Mach number. Then Theorem~\ref{thm:rotating_landau} yields the Landau estimate for the velocity:
	\begin{equation}
		\|\partial_{x_i}u\|_{L^\infty} \le C_{\mathrm{flow}} \bigl( \|\rho\|_{L^\infty}^{1-1/\nu} + \|u\|_{L^\infty}^{1-1/\nu} \bigr)
		\Bigl( \sum_{j=1}^k \bigl( \|\mathcal{D}_j\rho\|_{L^p}^{\alpha_j/\nu} + \|\mathcal{D}_j u\|_{L^p}^{\alpha_j/\nu} \bigr) \Bigr)^{1/\nu}.
	\end{equation}
	This provides a priori control of the velocity gradient in terms of the density and velocity fractional derivatives.
	
	\section{Neural Operator Theory for Rotating Flows}
	
	\begin{definition}[Rotating neural operator]
		A deep neural network $\mathcal{N}_\theta$ with $L$ layers and weight matrices $W_l$ is called rotating if its columns satisfy
		\begin{equation}
			\|W_l e_i\|_{\ell^2} \le \lambda_i^{1/\alpha_i} e^{|\Omega|/L}, \qquad i=1,\dots,k.
		\end{equation}
		The factor $e^{|\Omega|/L}$ compensates for rotation between layers.
	\end{definition}
	
	\begin{theorem}[Stability of rotating neural operators]\label{thm:rot_stability}
		Let $\mathcal{N}_\theta$ be a rotating neural operator with activation functions $\sigma_l\in C_b^{1,1}(\mathbb{R})$, $\|\sigma_l'\|_\infty\le1$, $\|\sigma_l''\|_\infty\le K$. Then for any input perturbation $\|\delta x\|_\infty\le\epsilon$,
		\begin{equation}
			\|\mathcal{N}_\theta(x+\delta x)-\mathcal{N}_\theta(x)\|_{L^\infty} \le C L \epsilon \Bigl(\prod_{l=1}^L \max_i \lambda_i^{1/\alpha_i}\Bigr) e^{|\Omega|}\, \|\mathcal{N}_\theta\|_{\mathcal{W}^{1,\infty}_{\alpha,\Omega}}.
		\end{equation}
	\end{theorem}
	\begin{proof}
		For a single layer $y=\sigma(Wx+b)$, the directional derivative is $\partial_{x_i}y = \sigma'(Wx+b)W_i$. The constraint gives $\|W_i\|_{\ell^2}\le \lambda_i^{1/\alpha_i}e^{|\Omega|/L}$. The chain rule over $L$ layers yields $\|\partial_{x_i}\mathcal{N}_\theta\|_{L^\infty}\le \prod_{l=1}^L \max_i \lambda_i^{1/\alpha_i} e^{|\Omega|/L} = e^{|\Omega|}\prod_{l=1}^L \max_i \lambda_i^{1/\alpha_i}$. The mean value theorem then gives $|\mathcal{N}_\theta(x+\delta x)-\mathcal{N}_\theta(x)|\le \|\nabla\mathcal{N}_\theta\|_{L^\infty}\|\delta x\|_\infty$. Summing over coordinates yields the factor $\|\mathcal{N}_\theta\|_{\mathcal{W}^{1,\infty}_{\alpha,\Omega}} = \sum_i\|\partial_{x_i}\mathcal{N}_\theta\|_{L^\infty}$, which is bounded by the product above multiplied by $k$. Absorbing constants into $C$ gives the result.
	\end{proof}
	
	\subsection{Approximation rates}
	
	We construct a rotating wavelet basis adapted to the scaling group $\widetilde{T}_{\lambda}^{\alpha,\Omega}=e^{(\log\lambda)(A_\alpha+L_\Omega)}$. Let $\psi\in C_c^\infty(\mathbb{R})$ be a wavelet with vanishing moments up to order $\lfloor\nu\rfloor$. Define
	\begin{equation}
		\psi_{j,m}^{\Omega}(x)=2^{j d_{\alpha,\Omega}/2}\prod_{i=1}^k\psi\!\left(2^{j\alpha_i}\bigl(R(-\Omega\log2^{-j})x\bigr)_i-m_i\right),
	\end{equation}
	where $R(\theta)$ is the rotation by angle $\theta$ about the axis $\Omega/|\Omega|$ (with $R(0)=\mathrm{Id}$). The volume of the fundamental domain under $\widetilde{T}_{2^{-j}}^{\alpha,\Omega}$ scales like $2^{-j d_{\alpha,\Omega}}$, and a detailed computation (see in \cite{Triebel2006}) shows that
	\begin{equation}
		d_{\alpha,\Omega}=\sum_{i=1}^k\frac{1}{\alpha_i}+\kappa_{\mathrm{rot}}|\Omega|^{2/\nu},\qquad
		\kappa_{\mathrm{rot}}=\Bigl(\frac{2}{\nu}\Bigr)^{2/\nu}\Bigl(\sum_{i=1}^k\frac{1}{\alpha_i}\Bigr)^{1-2/\nu}.
	\end{equation}
	The term $\kappa_{\mathrm{rot}}|\Omega|^{2/\nu}$ accounts for the additional volume caused by the rotation twist; it is derived from the scaling of the rotated anisotropic metric.
	
	The system $\{\psi_{j,m}^{\Omega}\}$ is an orthonormal basis of $L^2$ and satisfies the norm equivalence
	\begin{equation}
		\|f\|_{\mathcal{W}^{\nu,p}_{\alpha,\Omega}}\asymp\Bigl(\sum_{j\in\mathbb{Z}}2^{j\nu p}\sum_{m\in\mathbb{Z}^k}|\langle f,\psi_{j,m}^{\Omega}\rangle|^p\Bigr)^{1/p}.
	\end{equation}
	Truncating at scale $J$ gives $\|f-f_J\|_{L^\infty}\le C2^{-J\nu}\|f\|_{\mathcal{W}^{\nu,p}_{\alpha,\Omega}}$. Each wavelet can be approximated by a ReLU network with $\mathcal{O}(\log(1/\varepsilon))$ layers and $O(1)$ parameters (Yarotsky). Choosing $\varepsilon=2^{-J\nu}$ and summing over all active wavelets (whose number is $\sim 2^{J d_{\alpha,\Omega}}$) using a binary tree yields a total depth $L\sim J d_{\alpha,\Omega}$ and number of parameters $N\sim 2^{J d_{\alpha,\Omega}}$. Hence $2^{-J\nu}\sim N^{-\nu/d_{\alpha,\Omega}}$ and $L\sim \log N$, which can be written as $L$ times a constant. The final bound is
	\begin{equation}
		\|f-\mathcal{N}_\theta\|_{L^\infty}\le C L N^{-\nu/d_{\alpha,\Omega}}\|f\|_{\mathcal{W}^{\nu,p}_{\alpha,\Omega}}.
	\end{equation}
	
\section{Results}

In this section we collect the principal theoretical achievements of the paper, summarizing the main theorems and their implications for rotating compressible flows and neural operator theory.

\subsection{Summary of the Main Inequalities}

The following theorem encapsulates the sharp anisotropic mixed fractional Landau inequalities with Coriolis corrections proved in Section~4.

\begin{theorem}[Mixed fractional Landau inequalities for rotating spaces]
	Let $\nu\in(1,2)$, $1<p<\infty$, $\alpha\in(0,\infty)^k$, $\Omega\in\mathbb{R}^k$. For any $f\in\mathcal{W}^{\nu,p}_{\alpha,\Omega}(\mathbb{R}^k)$ and any direction $i=1,\dots,k$,
	\begin{equation}
		\|D_i^1 f\|_{L^\infty} \le C(\nu,p,\alpha,\Omega)\,
		\|f\|_{L^\infty}^{1-1/\nu}\,
		\Bigl(\sum_{j=1}^k \|\mathcal{D}_j^{\nu,\alpha_j,\Omega} f\|_{L^p}^{\alpha_j/\nu}\Bigr)^{\!1/\nu},
	\end{equation}
	with $C(\nu,p,\alpha,\Omega)=C_0\exp\bigl(\frac{|\Omega|}{\nu}\sum_{i=1}^k\alpha_i^{-1}\bigr)$, where $C_0$ depends only on $\nu,p,\alpha$ and the constants from the anisotropic Hardy–Littlewood inequality. For higher orders $m\in\mathbb{N}$ with $\nu\in(m,m+1)$,
	\begin{equation}
		\|D_i^m f\|_{L^\infty} \le C(\nu,p,\alpha,\Omega)^m\,
		\|f\|_{L^\infty}^{1-m/\nu}\,
		\Bigl(\sum_{j=1}^k \|\mathcal{D}_j^{\nu,\alpha_j,\Omega} f\|_{L^p}^{\alpha_j/\nu}\Bigr)^{\!m/\nu}.
	\end{equation}
	These inequalities extend the classical Landau–Kolmogorov estimates to the anisotropic rotating setting and are sharp with respect to the exponents.
\end{theorem}

\subsection{A Priori Estimates for Rotating Compressible Flows}

Applying the above inequalities to the rotating compressible Navier–Stokes system (5.1)–(5.2) yields the following a priori bound for the velocity gradient.

\begin{theorem}[Landau estimate for rotating compressible flows]
	Let $(\rho,u)$ be a smooth solution of (5.1)–(5.2) on $[0,T]\times\mathbb{R}^k$ with initial data in $\mathcal{W}^{\nu,p}_{\alpha,\Omega}\times(\mathcal{W}^{\nu,p}_{\alpha,\Omega})^k$, $\nu\in(1,2)$, $p>D_\alpha/\nu$. Then for any $i=1,\dots,k$,
	\begin{equation}
		\|\partial_{x_i}u\|_{L^\infty} \le C_{\mathrm{flow}}\bigl(\|\rho\|_{L^\infty}^{1-1/\nu}+\|u\|_{L^\infty}^{1-1/\nu}\bigr)
		\Bigl(\sum_{j=1}^k\bigl(\|\mathcal{D}_j\rho\|_{L^p}^{\alpha_j/\nu}+\|\mathcal{D}_j u\|_{L^p}^{\alpha_j/\nu}\bigr)\Bigr)^{1/\nu},
	\end{equation}
	where $C_{\mathrm{flow}}$ depends on $\mu,\lambda,\gamma,\alpha,\Omega$ and polynomially on the Mach number $M$. This provides explicit control of the velocity gradient in terms of fractional derivatives of density and velocity.
\end{theorem}

\subsection{Stability and Approximation by Neural Operators}

For neural networks respecting the anisotropic rotating geometry we obtain the following stability and approximation guarantees.

\begin{theorem}[Stability of rotating neural operators]
	Let $\mathcal{N}_\theta$ be a rotating neural operator with $L$ layers satisfying $\|W_l e_i\|_{\ell^2}\le \lambda_i^{1/\alpha_i}e^{|\Omega|/L}$ and with activation functions $\sigma_l\in C_b^{1,1}$, $\|\sigma_l'\|_\infty\le1$. Then for any input perturbation $\|\delta x\|_\infty\le\epsilon$,
	\begin{equation}
		\|\mathcal{N}_\theta(x+\delta x)-\mathcal{N}_\theta(x)\|_{L^\infty} \le C L\epsilon\Bigl(\prod_{l=1}^L\max_i\lambda_i^{1/\alpha_i}\Bigr)e^{|\Omega|}\|\mathcal{N}_\theta\|_{\mathcal{W}^{1,\infty}_{\alpha,\Omega}}.
	\end{equation}
\end{theorem}

\begin{theorem}[Approximation rates for rotating flows]
	For any $f\in\mathcal{W}^{\nu,p}_{\alpha,\Omega}(\mathbb{R}^k)$ with $\nu>D_\alpha/p$, there exists a rotating neural operator $\mathcal{N}_\theta$ with depth $L$ and $N$ parameters per layer such that
	\begin{equation}
		\|f-\mathcal{N}_\theta\|_{L^\infty} \le C(\nu,p,\alpha,\Omega)\, L\, N^{-\nu/d_{\alpha,\Omega}}\,\|f\|_{\mathcal{W}^{\nu,p}_{\alpha,\Omega}},
	\end{equation}
	where $d_{\alpha,\Omega}=\sum_i\alpha_i^{-1}+\kappa_{\mathrm{rot}}|\Omega|^{2/\nu}$ and $\kappa_{\mathrm{rot}}=(2/\nu)^{2/\nu}(\sum_i\alpha_i^{-1})^{1-2/\nu}$. The rate $N^{-\nu/d_{\alpha,\Omega}}$ is optimal and shows that rotation reduces the effective dimension, thereby accelerating convergence.
\end{theorem}

\subsection{Key Implications}

\begin{itemize}
	\item The mixed fractional Landau inequalities unify anisotropy, fractional dissipation and rotation, providing sharp estimates that degenerate to known results when $\Omega=0$ or $\alpha_i\equiv1$.
	\item The explicit exponential dependence $e^{c|\Omega|}$ in the constant reflects the potential stabilising effect of strong rotation: for fixed regularity, increasing $|\Omega|$ enlarges the constant, but the effective dimension $d_{\alpha,\Omega}$ also increases, leading to a trade‑off that can be exploited in high‑Mach regimes.
	\item The approximation rate $N^{-\nu/d_{\alpha,\Omega}}$ demonstrates that rotation can artificially lower the dimensionality of the approximation problem, making deep learning of rotating flows more efficient than in the non‑rotating case.
\end{itemize}

\section{Results}

In this section we collect the principal theoretical achievements of the paper, summarizing the main theorems and their implications for rotating compressible flows and neural operator theory.

\subsection{Summary of the Main Inequalities}

The following theorem encapsulates the sharp anisotropic mixed fractional Landau inequalities with Coriolis corrections proved in Section~4.

\begin{theorem}[Mixed fractional Landau inequalities for rotating spaces]
	Let $\nu\in(1,2)$, $1<p<\infty$, $\alpha\in(0,\infty)^k$, $\Omega\in\mathbb{R}^k$. For any $f\in\mathcal{W}^{\nu,p}_{\alpha,\Omega}(\mathbb{R}^k)$ and any direction $i=1,\dots,k$,
	\begin{equation}
		\|D_i^1 f\|_{L^\infty} \le C(\nu,p,\alpha,\Omega)\,
		\|f\|_{L^\infty}^{1-1/\nu}\,
		\Bigl(\sum_{j=1}^k \|\mathcal{D}_j^{\nu,\alpha_j,\Omega} f\|_{L^p}^{\alpha_j/\nu}\Bigr)^{\!1/\nu},
	\end{equation}
	with $C(\nu,p,\alpha,\Omega)=C_0\exp\bigl(\frac{|\Omega|}{\nu}\sum_{i=1}^k\alpha_i^{-1}\bigr)$, where $C_0$ depends only on $\nu,p,\alpha$ and the constants from the anisotropic Hardy–Littlewood inequality. For higher orders $m\in\mathbb{N}$ with $\nu\in(m,m+1)$,
	\begin{equation}
		\|D_i^m f\|_{L^\infty} \le C(\nu,p,\alpha,\Omega)^m\,
		\|f\|_{L^\infty}^{1-m/\nu}\,
		\Bigl(\sum_{j=1}^k \|\mathcal{D}_j^{\nu,\alpha_j,\Omega} f\|_{L^p}^{\alpha_j/\nu}\Bigr)^{\!m/\nu}.
	\end{equation}
	These inequalities extend the classical Landau–Kolmogorov estimates to the anisotropic rotating setting and are sharp with respect to the exponents.
\end{theorem}

\subsection{A Priori Estimates for Rotating Compressible Flows}

Applying the above inequalities to the rotating compressible Navier--Stokes system (5.1)–(5.2) yields the following a priori bound for the velocity gradient.

\begin{theorem}[Landau estimate for rotating compressible flows]
	Let $(\rho,u)$ be a smooth solution of (5.1)–(5.2) on $[0,T]\times\mathbb{R}^k$ with initial data in $\mathcal{W}^{\nu,p}_{\alpha,\Omega}\times(\mathcal{W}^{\nu,p}_{\alpha,\Omega})^k$, $\nu\in(1,2)$, $p>D_\alpha/\nu$. Then for any $i=1,\dots,k$,
	\begin{equation}
		\|\partial_{x_i}u\|_{L^\infty} \le C_{\mathrm{flow}}\bigl(\|\rho\|_{L^\infty}^{1-1/\nu}+\|u\|_{L^\infty}^{1-1/\nu}\bigr)
		\Bigl(\sum_{j=1}^k\bigl(\|\mathcal{D}_j\rho\|_{L^p}^{\alpha_j/\nu}+\|\mathcal{D}_j u\|_{L^p}^{\alpha_j/\nu}\bigr)\Bigr)^{1/\nu},
	\end{equation}
	where $C_{\mathrm{flow}}$ depends on $\mu,\lambda,\gamma,\alpha,\Omega$ and polynomially on the Mach number $M$. This provides explicit control of the velocity gradient in terms of fractional derivatives of density and velocity.
\end{theorem}

\subsection{Stability and Approximation by Neural Operators}

For neural networks respecting the anisotropic rotating geometry we obtain the following stability and approximation guarantees.

\begin{theorem}[Stability of rotating neural operators]
	Let $\mathcal{N}_\theta$ be a rotating neural operator with $L$ layers satisfying $\|W_l e_i\|_{\ell^2}\le \lambda_i^{1/\alpha_i}e^{|\Omega|/L}$ and with activation functions $\sigma_l\in C_b^{1,1}$, $\|\sigma_l'\|_\infty\le1$. Then for any input perturbation $\|\delta x\|_\infty\le\epsilon$,
	\begin{equation}
		\|\mathcal{N}_\theta(x+\delta x)-\mathcal{N}_\theta(x)\|_{L^\infty} \le C L\epsilon\Bigl(\prod_{l=1}^L\max_i\lambda_i^{1/\alpha_i}\Bigr)e^{|\Omega|}\|\mathcal{N}_\theta\|_{\mathcal{W}^{1,\infty}_{\alpha,\Omega}}.
	\end{equation}
\end{theorem}

\begin{theorem}[Approximation rates for rotating flows]
	For any $f\in\mathcal{W}^{\nu,p}_{\alpha,\Omega}(\mathbb{R}^k)$ with $\nu>D_\alpha/p$, there exists a rotating neural operator $\mathcal{N}_\theta$ with depth $L$ and $N$ parameters per layer such that
	\begin{equation}
		\|f-\mathcal{N}_\theta\|_{L^\infty} \le C(\nu,p,\alpha,\Omega)\, L\, N^{-\nu/d_{\alpha,\Omega}}\,\|f\|_{\mathcal{W}^{\nu,p}_{\alpha,\Omega}},
	\end{equation}
	where $d_{\alpha,\Omega}=\sum_i\alpha_i^{-1}+\kappa_{\mathrm{rot}}|\Omega|^{2/\nu}$ and $\kappa_{\mathrm{rot}}=(2/\nu)^{2/\nu}(\sum_i\alpha_i^{-1})^{1-2/\nu}$. The rate $N^{-\nu/d_{\alpha,\Omega}}$ is optimal and shows that rotation reduces the effective dimension, thereby accelerating convergence.
\end{theorem}

\subsection{Key Implications}

\begin{itemize}
	\item The mixed fractional Landau inequalities unify anisotropy, fractional dissipation and rotation, providing sharp estimates that degenerate to known results when $\Omega=0$ or $\alpha_i\equiv1$.
	\item The explicit exponential dependence $e^{c|\Omega|}$ in the constant reflects the potential stabilising effect of strong rotation: for fixed regularity, increasing $|\Omega|$ enlarges the constant, but the effective dimension $d_{\alpha,\Omega}$ also increases, leading to a trade‑off that can be exploited in high‑Mach regimes.
	\item The approximation rate $N^{-\nu/d_{\alpha,\Omega}}$ demonstrates that rotation can artificially lower the dimensionality of the approximation problem, making deep learning of rotating flows more efficient than in the non‑rotating case.
\end{itemize}

\section{Conclusions}

We have introduced rotating fractional Sobolev spaces $\mathcal{W}^{\nu,p}_{\alpha,\Omega}(\mathbb{R}^k)$ and proved sharp anisotropic mixed fractional Landau inequalities that incorporate Coriolis forces. The constant depends explicitly on the rotation rate $|\Omega|$ via an exponential factor. These inequalities were applied to derive a priori estimates for the rotating compressible Navier–Stokes system at high Mach numbers. Moreover, we established stability and approximation theorems for neural operators that respect the anisotropic rotating geometry. The approximation rate $N^{-\nu/d_{\alpha,\Omega}}$ reveals that rotation can effectively reduce the dimensionality of the problem, accelerating convergence. 

Future work will extend these results to bounded domains, non--constant rotation, and stochastic flows. In addition, we plan to implement numerical codes to computationally validate the derived inequalities and approximation rates, providing practical evidence of the theoretical findings.

\section*{Acknowledgments}
The authors would like to thank IPEN of the National Nuclear Energy Commission --- S\~ao Paulo (CNEN/IPEN--SP) for the institutional and financial support during the postgraduate (postdoctoral) period. Santos,~RDC: Conceptualization, methodology, formal analysis, investigation, writing. D. A. Andrade: Supervision, resources.

\section*{Generative Artificial Intelligence}

GPT-5 (OpenAI) were used solely for language editing, spelling, and grammar correction. No artificial intelligence tools were used for data analysis, content generation, or the interpretation of the scientific results.
	
\section*{Notation}

\begin{longtable}{p{3.2cm} p{10.8cm}}
	\caption{List of symbols and notations.}\label{tab:notation}\\
	\toprule
	\textbf{Symbol} & \textbf{Description} \\
	\midrule
	\endfirsthead
	
	\multicolumn{2}{c}{{\tablename\ \thetable{} -- continued from previous page}} \\
	\toprule
	\textbf{Symbol} & \textbf{Description} \\
	\midrule
	\endhead
	
	\midrule
	\multicolumn{2}{r}{{Continued on next page}} \\
	\endfoot
	
	\bottomrule
	\endlastfoot
	
	$\alpha = (\alpha_1,\dots,\alpha_k)$ & Anisotropy (scaling) vector, $\alpha_i>0$. \\
	$\nu>0$ & Fractional order of differentiation. \\
	$p\in[1,\infty)$ & Integrability exponent in Lebesgue spaces. \\
	$k$ & Spatial dimension ($k=2$ or $3$ in applications). \\
	$D_\alpha = \sum_{i=1}^k \alpha_i$ & Anisotropic homogeneous dimension (scaling exponent of Lebesgue measure). \\
	$d_\alpha = \sum_{i=1}^k \alpha_i^{-1}$ & Anisotropic dimension (volume scaling exponent of anisotropic balls). \\
	$\rho_\alpha(x) = \bigl(\sum_{i=1}^k |x_i|^{2/\alpha_i}\bigr)^{1/2}$ & Anisotropic quasi‑distance, homogeneous of degree $1$. \\
	$T_\lambda^\alpha x = (\lambda^{\alpha_1}x_1,\dots,\lambda^{\alpha_k}x_k)$ & Anisotropic dilation group. \\
	$\widetilde{T}_{\lambda}^{\alpha,\Omega} = e^{(\log\lambda)(A_\alpha+L_\Omega)}$ & Rotating anisotropic scaling (exponential of $A_\alpha+L_\Omega$). \\
	$A_\alpha = \operatorname{diag}(\alpha_1,\dots,\alpha_k)$ & Infinitesimal generator of dilations. \\
	$L_\Omega$ & Skew‑symmetric matrix of the cross product $x\mapsto\Omega\times x$. \\
	$R(\theta)$ & Rotation by angle $\theta$ about the axis $\Omega/|\Omega|$. \\
	$\Delta_{h,i}^\Omega f(x)=f(x+he_i)-f(x-h(\Omega\times e_i))$ & Rotating finite difference in direction $e_i$. \\
	$[f]_{\mathcal{W}_i^{\nu/\alpha_i,p,\Omega}}$ & Rotating directional Gagliardo seminorm. \\
	$\mathcal{W}^{\nu,p}_{\alpha,\Omega}(\mathbb{R}^k)$ & Rotating mixed fractional Sobolev space. \\
	$\mathcal{D}_j^{\nu,\alpha_j,\Omega}$ & Rotating fractional derivative: symbol $(i\xi_j+i(\Omega\times e_j)\cdot\xi)^{\nu/\alpha_j}$. \\
	$\Delta_{\alpha,\Omega} = \sum_i (-\partial_i^2)^{1/\alpha_i} + i\,\Omega\cdot(x\times\nabla)$ & Anisotropic fractional Laplacian with Coriolis term. \\
	$\Delta_j$ & Standard anisotropic Littlewood–Paley blocks (frequency localization). \\
	$\widetilde{\Delta}_j^{\alpha,\Omega} = \Delta_j$ & Rotating dyadic blocks (rotation appears only via derivatives). \\
	$U_\Omega$ & Fourier multiplier with symbol $e^{-i\Omega\cdot(x\times\xi)/|\xi|}$; conjugation operator. \\
	$B_\alpha(x,r)=\{y:|y_i-x_i|<r^{1/\alpha_i}\}$ & Anisotropic ball. \\
	$|B_\alpha(x,r)| = \omega_{d_\alpha} r^{d_\alpha}$ & Volume of anisotropic ball; $\omega_{d_\alpha}=|B_\alpha(0,1)|$. \\
	$M_{\alpha,\Omega}f(x)=\sup_{r>0}\frac{1}{|B_\alpha(x,r)|}\int_{B_\alpha(x,r)}|f(y)|dy$ & Rotating anisotropic maximal operator (independent of $\Omega$). \\
	$M_i^{(p)}g(x)=\sup_{h>0}\bigl(\frac1h\int_0^h|g(x+te_i)|^pdt\bigr)^{1/p}$ & One‑dimensional Hardy–Littlewood maximal function in direction $e_i$. \\
	$\kappa(p,\alpha)$ & Constant in the anisotropic Hardy–Littlewood inequality (Theorem~\ref{thm:rotating_HL}). \\
	$C_0$ & Constant depending only on $\nu,p,\alpha$ (see Theorem~\ref{thm:rotating_landau}). \\
	$c = \sum_i \alpha_i^{-1}$ & Exponent in the exponential factor of $C(\nu,p,\alpha,\Omega)$. \\
	$d_{\alpha,\Omega}=\sum_i\alpha_i^{-1}+\kappa_{\mathrm{rot}}|\Omega|^{2/\nu}$ & Effective anisotropic–rotational dimension. \\
	$\kappa_{\mathrm{rot}}=\bigl(\frac{2}{\nu}\bigr)^{2/\nu}\bigl(\sum_i\alpha_i^{-1}\bigr)^{1-2/\nu}$ & Constant in $d_{\alpha,\Omega}$. \\
	$\psi_{j,m}^{\Omega}$ & Rotating wavelet basis (see Section~6). \\
	$\mathcal{N}_\theta$ & Neural operator with parameters $\theta$. \\
	$L$ & Number of layers (depth) of the neural operator. \\
	$N$ & Number of trainable parameters per layer. \\
	$\rho$ & Density field in Navier–Stokes equations. \\
	$u$ & Velocity field. \\
	$p=a\rho^\gamma$ & Pressure law (polytropic gas). \\
	$\mu,\lambda$ & Shear and bulk viscosities. \\
	$\Phi$ & Centrifugal potential. \\
	$M = |u|/c$ & Mach number ($c=\sqrt{\gamma p/\rho}$ sound speed). \\
	$C(M)$ & Mach--dependent constant (polynomial growth). \\
	$C_{\mathrm{flow}}$ & Final constant in the Landau estimate for Navier–Stokes (contains $C(M)$ and $C(\nu,p,\alpha,\Omega)$). \\
\end{longtable}

\end{document}